\documentclass[11pt]{article}
\usepackage{a4wide}

\usepackage[mathscr]{euscript}

\usepackage{amsfonts,amsthm}
\usepackage{amssymb}
\usepackage{amsmath}

\usepackage{tikz-cd}
\usetikzlibrary{arrows}
\usepackage{mathabx}
\usepackage{sseq}
\usepackage{rotating}
\usepackage{bbm}
\usepackage[all,cmtip]{xy}
\usepackage{amscd}

\numberwithin{equation}{section}

\usepackage{caption}
\usepackage{subcaption}

\usepackage{tikz}
\usepackage{enumitem}

\usepackage{mathrsfs}
\usepackage{color}
\usepackage{xcolor}

\usepackage{graphicx}
\usepackage{dsfont}

\usepackage[colorlinks]{hyperref}

\newtheorem{theorem}{Theorem}[section]

\newtheorem{lemma}[theorem]{Lemma}
\newtheorem{proposition}[theorem]{Proposition}

\theoremstyle{remark}
\newtheorem{remark}[theorem]{Remark}

\theoremstyle{definition}

\theoremstyle{theorem}
\newtheorem{TheoremA}{Theorem}

\newtheorem*{maintheorem}{Main Theorem}

\newcommand{\ow}{\omega}
\newcommand{\p}{\partial}

\newcommand{\C}{{\mathbb{C}}}
\newcommand{\I}{{\mathcal{I}}}
\newcommand{\A}{{\mathcal{A}}}
\newcommand{\R}{{\mathbb{R}}}

\newcommand{\Z}{{\mathbb{Z}}}

\renewcommand{\epsilon}{\varepsilon}
\renewcommand{\theta}{\vartheta}

\DeclareMathOperator{\PSL}{PSL}

\DeclareMathOperator{\dR}{dR}
\DeclareMathOperator{\std}{std}

\DeclareMathOperator{\Ham}{Ham}

\DeclareMathOperator{\ev}{ev}
\DeclareMathOperator{\pr}{pr}

\DeclareMathOperator{\pt}{pt}

\DeclareMathOperator{\Diff}{Diff}

\DeclareMathOperator{\id}{id}

\DeclareMathOperator{\Symp}{Symp}

\DeclareMathOperator{\Fix}{Fix}

\DeclareMathOperator{\SO}{SO}

\usepackage{graphicx}
\usepackage{authblk}

\makeatletter
\newcommand{\subjclass}[2][2010]{%
  \let\@oldtitle\@title%
  \gdef\@title{\@oldtitle\footnotetext{#1 \emph{Mathematics Subject Classification.} #2.}}%
}
\newcommand{\keywords}[1]{%
  \let\@@oldtitle\@title%
  \gdef\@title{\@@oldtitle\footnotetext{\emph{Key words and phrases.} #1.}}%
}
\newcommand{\Date}[1]{%
  \let\@@@oldtitle\@title%
\gdef\@title{\@@@oldtitle\footnotetext{\emph{Date.} #1.}}%
}
\makeatother

\begin{document}
\title{{\bf \Large
Anti-symplectic involutions for Lagrangian spheres in a symplectic quadric surface
}}

\author{Joontae Kim and Jiyeon Moon}
\date{}

\maketitle

\begin{abstract}
We show that the space of anti-symplectic involutions of a monotone $S^2\times S^2$ whose fixed points set is a Lagrangian sphere is connected.
This follows from a stronger result, namely that any two anti-symplectic involutions in that space are Hamiltonian isotopic.

\end{abstract}


\section{Introduction and main results}
A \emph{symplectic manifold} $(M,\ow)$ is an even dimensional smooth manifold $M$ equipped with a closed non-degenerate 2-form $\ow$. 
Recall that an \emph{anti-symplectic involution} of a symplectic manifold $(M,\ow)$ is a diffeomorphism $R\in \Diff(M)$ satisfying $R^2=\id_M$ and $R^*\ow=-\ow$.
Its fixed point set $\Fix(R) = \{ x\in M \mid R(x)=x\}$ is a Lagrangian submanifold, that is, $\dim \Fix(R)=\frac{1}{2}\dim M$ and $\ow$ vanishes on the tangent bundle of $\Fix(R)$, if it is non-empty.
In this case, $\Fix(R)$ is called a \emph{real Lagrangian}.
We abbreviate by $\A(M,\ow)$ the \emph{space of anti-symplectic involutions} of $(M,\ow)$.
The \emph{space of smooth involutions of a smooth manifold} $M$ is denoted by $\I(M)$.
Throughout the paper, $\Diff(M)$ and all of its subspaces are endowed with the $C^\infty$-topology.
It is a prominent question to investigate the injectivity of the natural forgetful map $\pi_0\A(M,\ow)\to \pi_0\I(M)$ in a \emph{symplectic ruled surface}, the total space of an $S^2$-fibration over a closed oriented surface with a symplectic form that is non-degenerate on the fibres.
This formulates a version of the \emph{symplectic isotopy conjecture} \cite[Problem~14 in Section~14.2]{MS} about the symplectomorphism group of ruled surfaces.
In particular, if the map is injective, then any two elements in $\A(M,\ow)$ are isotopic if and only if they are smoothly isotopic.

In this paper, we explore the question when $(M,\ow)$ is the \emph{monotone symplectic quadric surface} $(Q,\ow)=(S^2\times S^2,\ow_0\oplus\ow_0)$, where $\ow_0$ is a Euclidean area form normalized by $\int_{S^2}\ow_0=1$, and when the fixed point set of $R\in \A(Q,\ow)$ is a Lagrangian \emph{sphere}.
Based on foliations by $J$-holomorphic curves, Gromov \cite{Gromov} showed that the symplectomorphism group $\Symp(Q,\ow)$ deformation retracts onto the isometry group $(\SO(3)\times \SO(3))\rtimes \Z/2$, where $\Z/2$ acts by interchanging two $S^2$-factors, and hence the natural map $\pi_0\Symp(Q,\ow)\to \pi_0\Diff(Q)$ is injective, whence the symplectic isotopy conjecture holds in this case.
Hind \cite{HindS2S2} proved that every Lagrangian sphere in $Q$ is Hamiltonian isotopic to the \emph{antidiagonal} $\overline{\Delta}=\{(x,-x)\mid x\in S^2\}$.
Therefore, any two real Lagrangian spheres $L_0=\Fix(R_0)$ and $L_1=\Fix(R_1)$ in $Q$ are Hamiltonian isotopic, and it is natural to ask whether the corresponding involutions $R_0$ and $R_1$ are Hamiltonian isotopic.
We say that $R_0,R_1\in \A(M,\ow)$ are \emph{Hamiltonian isotopic} if there exists $\phi\in \Ham(M,\ow)$ such that $\phi^*R_1 := \phi^{-1}\circ R_1\circ\phi = R_0$, where $\Ham(M,\ow)$ denotes the \emph{group of Hamiltonian diffeomorphisms} of $M$.
The classification of Hamiltonian isotopy classes of anti-symplectic involutions is \emph{a priori} finer than the one of real Lagrangians, see Figure~\ref{fig: 1}.
We write $R_{\overline{\Delta}}(x,y)=(-y,-x)$ for the anti-symplectic involution of $Q$ with $\Fix(R_{\overline{\Delta}})=\overline{\Delta}$.
The main result of the paper is rather modest.
\begin{maintheorem}
Every anti-symplectic involution of $(S^2\times S^2,\ow_0\oplus\ow_0)$ whose fixed point set is a Lagrangian sphere is Hamiltonian isotopic to $R_{\overline{\Delta}}$.
\end{maintheorem}
In particular, any two anti-symplectic involutions of $Q$ whose fixed point set is a Lagrangian sphere are Hamiltonian isotopic.
We abbreviate by $\A(Q,\ow,S^2) = \{R\in \A(Q,\ow)\mid \Fix(R) \overset{\text{diff}}{\cong} S^2\}$ and 
$\I(Q,S^2) = \{\sigma\in \I(Q) \mid \Fix(\sigma)\overset{\text{diff}}{\cong} S^2\}$.
As an immediate corollary, we obtain a sort of a symplectic phenomenon.
\begin{TheoremA}
	$\pi_0\A(Q,\ow,S^2)=0$, while $\pi_0\I(Q,S^2)\ne 0$.
\end{TheoremA}
That $\pi_0\I(Q,S^2)\ne 0$ easily follows since the actions of $R_{\overline{\Delta}}$ and $\tau(x,y)=(y,x)$ on homology are different, that is, $(R_{\overline{\Delta}})_*\ne \tau_*$ on $H_2(Q)$, and hence they form different components in $\I(Q,S^2)$.
In consequence, the natural map $\pi_0\A(Q,\ow,S^2)\to \pi_0\I(Q,S^2)$ is injective, but not surjective.
It is known that $\A(Q,\ow)$ has at least four components for topological reasons; two of them have no fixed points (but the corresponding quotient manifolds are not diffeomorphic), and the fixed point sets of the remaining two cases are $S^2$ and $T^2$, see \cite[6.11.7~and~6.11.8]{Enriques} for details. 
By \cite{Kim2}, every real Lagrangian $T^2$ in $Q$ is Hamiltonian isotopic to the \emph{Clifford torus}, defined as the product of the equators in each $S^2$-factor.
Hence, there is no obvious reason to find non-isotopic anti-symplectic involutions of $Q$ having $T^2$ as the fixed point set.
\begin{remark}
We refer to an intriguing work by Kharlamov--Shevchishin \cite{KS} about anti-symplectic involutions in rational symplectic 4-manifolds.
The scheme of their proof can be adapted to obtain our main theorem.
It can be used to show the connectedness of the space of anti-symplectic involutions of $Q$ corresponding to other topological types.
Nevertheless, our approach is still interesting in its own right since this is more explicit and direct. 
\end{remark}
\paragraph{An outline of the proof.}
We first recall the proof of the result by Hind \cite{HindS2S2}.
A crucial ingredient is the following non-trivial statement.
\begin{theorem}[Hind]\label{thm: Hind}\label{thm: hind}
	Let $L$ be a Lagrangian sphere in $Q$. Then there exists a tame almost complex structure $J$ on $Q$ such that each leaf of the Gromov's foliations $\mathcal{F}_1$ and $\mathcal{F}_2$ associated to $J$ intersects $L$ transversely at a single point.
\end{theorem}
See Section~\ref{sec: gromovfoliation} for the definition of Gromov foliation theorem.
If $L=\overline{\Delta}$, then the Gromov's foliations associated to the standard complex structure $J_{\std}=j_{\C P^1}\oplus j_{\C P^1}$ satisfies the desired properties in Theorem~\ref{thm: hind}.
Since any Lagrangian sphere $L$ in a symplectic 4-manifold has self-intersection number $-2$, the homology class $[L]$ is equal to one of $\pm (A_1-A_2)\in H_2(Q)$, where $A_1=[S^2\times \{\pt\}]$ and $A_2=[\{\pt\}\times S^2]$ are generators.
Hence, the intersection number of $L$ and any Gromov's leaf is equal to $\pm1$, and they intersect minimally and transversely under the situation of Theorem~\ref{thm: Hind}.
With this understood, any orientation preserving diffeomorphism from $L$ to $\overline{\Delta}$ extends uniquely to a self-diffeomorphism of $Q$ by sending the leaves of the Gromov's foliations corresponding to $J$ to the ones from $J_{\std}$.
Then Moser isotopy furnishes a symplectomorphism acting trivially on homology, which turns out to be a Hamiltonian diffeomorphism by the celebrated result of Gromov.
Therefore, any Lagrangian sphere $L$ is mapped to $\overline{\Delta}$ by a Hamiltonian diffeomorphism.


When $L=\Fix(R)$ is a real Lagrangian sphere with $R\in \A(Q,\ow)$, we can prove Theorem~\ref{thm: Hind} using \emph{$R$-anti-invariant} tame almost complex structures, see Proposition~\ref{prop: real_hind}.
This allows us to construct a diffeomorphism of $Q$ which is equivariant with respect to the involutions $R$ and $R_{\overline{\Delta}}$. 
Finally, one can still apply the Moser trick in an equivariant way to obtain $\Phi\in \Ham(Q,\ow)$ satisfying $R=\Phi^*R_{\overline{\Delta}}$.
In contrast to the original work by Hind, we do not need a sophisticated method, namely SFT, to obtain Theorem~\ref{thm: Hind}.
Instead, the \emph{classical} methods in $J$-holomorphic curves theory are totally enough.

\begin{figure}[h]
\begin{center}
\begin{tikzpicture}[scale=0.85]
\node[blue] at (-4.1,0) {$R$};
\node[red] at (-4.1,-1.5) {$\Fix(R)$};
\draw[blue] [-implies,double equal sign distance] (1.2,0) -- (1.6, 0);
\draw[blue] [-implies,double equal sign distance] (5.8,0) -- (6.2, 0);
\draw[red] [-implies,double equal sign distance] (1.2,-1.5) -- (1.6, -1.5);
\draw[red] [-implies,double equal sign distance] (5.8,-1.5) -- (6.2, -1.5);
\draw [-implies,double equal sign distance] (-1,-0.5) -- (-1, -1);
\draw [-implies,double equal sign distance] (3.7,-0.5) -- (3.7, -1);
\draw [-implies,double equal sign distance] (8,-0.5) -- (8, -1);
\node[blue] at (-1,0) {Hamiltonian isotopic};
\node[red] at (-1,-1.5) {Hamiltonian isotopic};
\node[blue] at (3.7, 0) {isotopic};
\node[red] at (3.7, -1.5) {Lagrangian isotopic};
\node[blue] at (8.1, 0) {smoothly isotopic};
\node[red] at (8.1, -1.5) {smoothly isotopic};
\draw (-5,0.5)--(10,0.5)--(10,-2)--(-5,-2)--(-5,0.5);
\end{tikzpicture}
\end{center}
\vspace{-0.5cm}
\caption{Relations between isotopies of $R$ and $\Fix(R)$}
\label{fig: 1}
\end{figure}
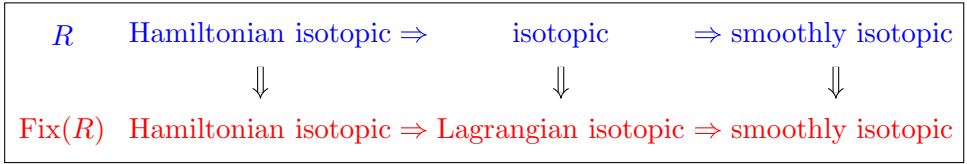
\vspace{-0.7cm}

\section{Proof of the Main Theorem}\label{sec: proof}
\subsection{Recollection of Gromov's foliation theorem}\label{sec: gromovfoliation}
Let $\mathcal{J}$ be the space of tame almost complex structures on the monotone symplectic quadric $Q$, which is non-empty and contractible \cite[Theorem~2.6.4]{MS}.
Fix $J\in \mathcal{J}$ and consider the \emph{moduli space of $J$-holomorphic spheres in the homology class $A_i$},
$$
\mathcal{M}(A_i,J)=\{u\colon \C P^1\to Q\mid \bar{\p}_J(u)=0,\ u_*[\C P^1]=A_i\}/\PSL(2,\C),
$$
where $\bar{\p}_J(u)=\frac{1}{2}(du+J\circ du\circ j_{\C P^1})\in \Omega^{0,1}(\C P^1,u^*TQ)$ denotes the complex anti-linear part of $du$.
Since every non-constant $J$-holomorphic sphere in $Q$ has positive Chern number and $A_i$ are primitive classes, every $u\in \mathcal{M}(A_i,J)$ is not multiply covered, and hence is simple, see \cite[Lemma~9.4.6]{McSalJholo}.
From the adjunction formula \cite[Theorem~2.6.4 or Corollary~E.1.7]{McSalJholo},
$$
A_i\bullet A_i-\langle c_1(Q),A_i \rangle +\chi(\C P^1)=0,
$$
which is a numerical criterion for simple $J$-holomorphic curves to be embedded, every $u\in \mathcal{M}(A_i,J)$ is embedded.
It follows from the automatic transversality, see \cite[Theorem~1]{HSL} or \cite[Lemma~3.3.3]{McSalJholo}, that for every $J\in \mathcal{J}$ the moduli space $\mathcal{M}(A_i,J)$ is a smooth manifold of dimension
$$
\dim_\C Q\cdot \chi(\C P^1)+2\langle c_1(Q), A_i\rangle - \dim_\R \PSL(2,\C) = 2.
$$
Since every $J$-holomorphic sphere in the homology class $A_i$ has the minimal possible positive energy, Gromov's compactness theorem asserts that $\mathcal{M}(A_i,J)$ is compact.
By positivity of intersections \cite[Section~2.6]{McSalJholo}, any two $J$-holomorphic spheres in $\mathcal{M}(A_1,J)$ and $\mathcal{M}(A_2,J)$ intersect transversely at a single point as $A_1\bullet A_2=1$.
Let $\widehat{\mathcal{M}}(A_i,J)$ be the space of parametrized $J$-holomorphic spheres in $Q$ representing $A_i$.
Since the evaluation map
\begin{equation}\label{eq: eval_map}
\ev_i\colon \mathcal{M}_1(A_i,J):=\widehat{\mathcal{M}}(A_i,J)\times_{\PSL(2,\C)}\C P^1 \longrightarrow Q,\qquad \ev\big([u,z]\big)=u(z)	
\end{equation}
is a diffeomorphism \cite[Proposition~9.4.4]{McSalJholo}, for each $i=1,2$ there exists a unique curve $\Sigma_i(\boldsymbol{x})$ in $\mathcal{M}(A_i,J)$ passing through any given point $\boldsymbol{x}\in Q$.
Therefore, $\mathcal{F}_1=\mathcal{M}(A_1,J)$ and $\mathcal{F}_2=\mathcal{M}(A_2,J)$ determine two transversal foliations of $Q$ whose leaves are embedded $J$-holomorphic spheres in the homology classes $A_1$ and $A_2$, respectively.
We call $\mathcal{F}_1$ and $\mathcal{F}_2$ the \emph{Gromov's foliations} associated to $J$.
We let $\mathcal{F}_i^{\std}$ be the Gromov's foliations associated to the standard complex structure $J_{\std}=j_{\C P^1}\oplus j_{\C P^1}$.
Analogously, their leaves passing through $\boldsymbol{x}\in Q$ are denoted by $\Sigma_i^{\std}(\boldsymbol{x})\in \mathcal{F}_i^{\std}$.

\subsection{A real analogue of Theorem~\ref{thm: Hind}}
We begin with the following homological result, see \cite[Section~2.1]{Kim2} for the proof.
\begin{lemma}\label{lem: inducedmap_invol}
For any $R\in \A(Q,\ow)$ the map $R_*$ induced in $H_2(Q)\cong \Z^2$ satisfies either
	$$
	C_1=\begin{pmatrix}
		0 & -1 \\
		-1 & 0
	\end{pmatrix}
	\quad\text{or}\quad
	C_2=\begin{pmatrix}
		-1 & 0 \\
		0 & -1
	\end{pmatrix}.
	$$
	Moreover, the following holds:
	\begin{enumerate}
		\item[(i)] $R_*=C_1$ if and only if $\Fix(R)$ is diffeomorphic to $S^2$.
		\item[(ii)] $R_*=C_2$ if and only if $\Fix(R)$ is either empty or diffeomorphic to $T^2$.
	\end{enumerate}
\end{lemma}
For $R\in \A(Q,\ow)$ we abbreviate by $\mathcal{J}_R=\{J\in \mathcal{J} \mid J=-R^*J\}$ the \emph{space of tame almost complex structures on $Q$ which are $R$-anti-invariant}.
Note that $\mathcal{J}_R$ is non-empty and contractible by S\'{e}vennec's arguments \cite[Proposition~1.1]{Welsch}.
One can directly check that the anti-symplectic involution $R_{\overline{\Delta}}$ sends a leaf in $\mathcal{F}_1^{\std}$ to a leaf in $\mathcal{F}_2^{\std}$, but it reverses the complex orientation on the leaves.
More generally, we observe the following.
\begin{lemma}\label{lem: involution_spheres}
Let $R\in \A(Q,\ow)$ such that $\Fix(R)$ is a Lagrangian sphere and let $J\in \mathcal{J}_R$. If $\Sigma_1\in \mathcal{F}_1$, then $R(\overline{\Sigma}_1)\in \mathcal{F}_2$. Similarly, if $\Sigma_2\in \mathcal{F}_2$, then $R(\overline{\Sigma}_2)\in \mathcal{F}_1$.
\end{lemma}
\begin{proof}
Let $u\colon \C P^1\to Q$ be a $J$-holomorphic sphere in the homology class $A_1$, which parametrizes $\Sigma_1\in \mathcal{F}_1$.
Pick any anti-holomorphic involution $\rho$ of $\C P^1$ which is necessarily orientation-reversing.
We then verify that
\begin{align*}
\bar{\p}_J(R\circ u\circ \rho) &= dR \circ \bar{\p}_J(u) \circ d\rho = 0 \\
(R\circ u \circ \rho)_*[\C P^1] &= R_*(-u_*[\C P^1]) = R_*(-A_1) = A_2.	
\end{align*} 
Here the last equality follows from Lemma~\ref{lem: inducedmap_invol}.
We have seen that $R\circ u\circ \rho$ is a $J$-holomorphic sphere representing the class $A_2$, that is, $R(\overline{\Sigma}_1)\in \mathcal{F}_2$.
The same argument also works for a leaf $\Sigma_2\in \mathcal{F}_2$.
\end{proof}

Now we prove Theorem~\ref{thm: Hind} using $R$-anti-invariant tame almost complex structures.
The proof is quite elementary.
\begin{proposition}\label{prop: real_hind}
	Let $L=\Fix(R)$ be a real Lagrangian sphere in $Q$ and $J\in \mathcal{J}_R$. Then each leaf of the Gromov's foliations $\mathcal{F}_1$ and $\mathcal{F}_2$ associated to $J$ intersects $L$ transversely at a single point.
\end{proposition}
\begin{proof}
Without loss of generality, we consider the case when a leaf $\Sigma_1\in \mathcal{F}_1$ is given. Since the algebraic intersection number of $L$ and $\Sigma_1$ is non-zero, more precisely, $[L]\bullet [\Sigma_1] = \pm(A_1-A_2)\bullet A_1=\mp 1$, it follows that $L$ and $\Sigma_1$ must have an intersection point.
To see that $\Sigma_1$ and $L$ has a \emph{unique} intersection point, let $\boldsymbol{x},\boldsymbol{y}\in \Sigma_1\cap L$ be given.
By Lemma~\ref{lem: involution_spheres}, we can consider another leaf $\Sigma_2:=R(\overline{\Sigma}_1)$ of $\mathcal{F}_2$.
Since $\boldsymbol{x},\boldsymbol{y}$ are fixed points of $R$, we observe that $\boldsymbol{x},\boldsymbol{y}\in \Sigma_2\cap L$, and hence $\boldsymbol{x},\boldsymbol{y}\in \Sigma_1\cap \Sigma_2$. 
This implies that $\boldsymbol{x}=\boldsymbol{y}$ as claimed.
It remains to check that $\Sigma_1$ intersects $L$ \emph{transversely} at a single point $\boldsymbol{x}$. Let $X\in T_{\boldsymbol{x}}\Sigma_1\cap T_{\boldsymbol{x}}L$ be given.
Since $T_{\boldsymbol{x}}L=\{X\in T_{\boldsymbol{x}}Q\mid dR(X)=X\}$ and $dR(T_{\boldsymbol{x}}\Sigma_1)=T_{\boldsymbol{x}}\Sigma_2$, we deduce that $X=dR(X)\in T_{\boldsymbol{x}}\Sigma_2\cap T_{\boldsymbol{x}}L$, 
which implies $X\in T_{\boldsymbol{x}}\Sigma_1\cap T_{\boldsymbol{x}}\Sigma_2$.
Since $\Sigma_1$ and $\Sigma_2$ intersect transversely, we must have $X=0$.
Therefore, $\Sigma_1$ and $L$ intersect transversely for dimensional reasons.
\end{proof}

\subsection{A diffeomorphism of $Q$ induced by transversal foliations}
Let $R\in \A(Q,\ow)$ such that $L=\Fix(R)$ is a Lagrangian sphere.
We shall construct an explicit diffeomorphism $\Phi$ of $Q$ which is equivariant with respect to the anti-symplectic involutions $R$ and $R_{\overline{\Delta}}$.

Choose any diffeomorphism $\phi\colon L\to \overline{\Delta}$ which is orientation-preserving.
Pick $J\in\mathcal{J}_R$ and write $\mathcal{F}_1$ and $\mathcal{F}_2$ for the corresponding Gromov's foliations.
It follows from Proposition~\ref{prop: real_hind} that we can extend $\phi$ to a unique diffeomorphism $\Phi\colon Q\to Q$ determined by sending the leaves of the foliations $\mathcal{F}_1$ and $\mathcal{F}_2$ to the leaves of the standard foliations $\mathcal{F}^{\std}_1$ and $\mathcal{F}^{\std}_2$.
To express $\Phi$ explicitly, for $i=1,2$ we define the map 
$$
\pi_i\colon Q \longrightarrow L,\qquad \pi_i(\boldsymbol{x})=\Sigma_i(\boldsymbol{x})\cap L,
$$
which sends a leaf in $\mathcal{F}_i$ to its unique intersection point with $L$.
Since every leaf in $\mathcal{F}_i$ and $L$ intersect transversely, $\pi_i$ is smooth, see Remark~\ref{rem: smooth}.
The diffeomorphism $\Phi$ of $Q$ is then defined by
\begin{equation}\label{eq: phi_map}
\Phi(\boldsymbol{x}) := \Sigma_1^{\std}\big(\phi\pi_1(\boldsymbol{x})\big)\cap \Sigma^{\std}_2\big(\phi\pi_2(\boldsymbol{x})\big).
\end{equation}
whose inverse is given by
	$$
	\Phi^{-1}(\boldsymbol{x})=\Sigma_1(\phi^{-1}\pi_1^{\std}(\boldsymbol{x}))\cap \Sigma_2(\phi^{-1}\pi_2^{\std}(\boldsymbol{x})).
	$$
Here $\pi_i^{\std}$ are defined analogously using $J_{\std}$.
Since $\Phi$ preserves the leaves of $\mathcal{F}_i$ and $\mathcal{F}_i^{\std}$ and $\phi$ is orientation preserving, $\Phi$ acts trivially on homology, that is, the induced map $\Phi_*$ on $H_*(Q)$ is the identity.
\begin{remark}\label{rem: smooth}
The discussion below is somewhat technical and implicit in \cite{HindS2S2}, but for the sake of completeness, we explain that $\Phi$ and $\Phi^{-1}$ are \emph{smooth} following the arguments in the proof of \cite[Theorem~9.4.7]{McSalJholo}.
In that proof, it is shown that for any $J\in \mathcal{J}$ the maps 
\begin{align*}
\begin{alignedat}{11}
		\Lambda\colon &Q\longrightarrow Q,\qquad &\Lambda(\boldsymbol{x})&=\Sigma_1(\boldsymbol{x})\cap \Sigma_2(\boldsymbol{x})\\
		\sigma_i\colon &Q\longrightarrow \mathcal{M}(A_i,J),\qquad &\sigma_i(\boldsymbol{x})&=\Sigma_i(\boldsymbol{x})
\end{alignedat}
\end{align*}
are smooth.
It thus suffices to show that $\pi_i$ are smooth since then $
\Phi =\Lambda\circ (\phi\times \phi) \circ (\pi_1\times \pi_2)$ is smooth as well.
Consider the map $\eta_i\colon \mathcal{M}(A_i,J)\to L$ defined by $\eta_i(\Sigma_i)=\Sigma_i\cap L$.
By Proposition~\ref{prop: real_hind}, it is well-defined and bijective.
From the expression $\pi_i=\eta_i\circ \sigma_i$, it remains to check that $\eta_i$ is smooth.
Let $\pr_i\colon \mathcal{M}_1(A_i,J)\to \mathcal{M}(A_i,J)$ denote the projection, where $\mathcal{M}_1(A_i,J)$ is defined in \eqref{eq: eval_map}.
Since the evaluation map $\ev_i$ in \eqref{eq: eval_map} is a diffeomorphism, $\eta_i^{-1}=\pr_i\circ \ev_i|_L^{-1}$ is smooth.
From the splitting $T_{\boldsymbol{x}}L\oplus T_{\boldsymbol{x}}\Sigma_i(\boldsymbol{x})=T_{\boldsymbol{x}}Q$ for each $\boldsymbol{x}\in L$, we deduce that the differential of $\eta_i^{-1}$ at every point of $L$ is bijective.
Therefore, $\eta_i$ is a diffeomorphism by the inverse function theorem.
This shows that $\Phi$ is smooth, and the same argument also yields that $\Phi^{-1}$ is smooth.
\end{remark}
By virtue of the choice of $J$ in $\mathcal{J}_R$, the following observation is straight forward.
\begin{lemma}\label{lem: Phi_equiv}
	The diffeomorphism $\Phi\colon Q\to Q$ defined in \eqref{eq: phi_map} satisfies 
\begin{equation}\label{eq: Phi_equiv}
		\Phi \circ R = R_{\overline{\Delta}} \circ \Phi.
\end{equation}
\end{lemma}
\begin{proof}
Using Lemma~\ref{lem: involution_spheres}, for all $\boldsymbol{x}\in Q$ we have
\begin{itemize}
	\item $\Sigma_1(\boldsymbol{x})=R\big(\overline{\Sigma}_2(R(\boldsymbol{x}))\big)$. 
	\item $\Sigma_1^{\std}(\boldsymbol{x})=R_{\overline{\Delta}}\big(\overline{\Sigma}_2^{\std}(R_{\overline{\Delta}}(\boldsymbol{x}))\big)$.
	\item $\pi_1(\boldsymbol{x})=\pi_2(R(\boldsymbol{x}))$.
\end{itemize}
Here the last item follows from
$$
\pi_1(\boldsymbol{x}) = \Sigma_1(\boldsymbol{x})\cap L = R\big(\overline{\Sigma}_1(\boldsymbol{x})\cap L\big) = R \big(\overline{\Sigma}_1(\boldsymbol{x})\big) \cap L = \overline{\Sigma}_2\big(R(\boldsymbol{x})\big)\cap L = \pi_2(R(\boldsymbol{x})).
$$
We readily see that
\begin{align*}
	\Phi(R(\boldsymbol{x})) &= \Sigma_1^{\std}\big(\phi\pi_1(R(\boldsymbol{x}))\big)\cap \Sigma_2^{\std}\big(\phi\pi_2(R(\boldsymbol{x}))\big) \\
	&= \Sigma_1^{\std}\big(\phi\pi_2(\boldsymbol{x})\big)\cap \Sigma_2^{\std}\big(\phi\pi_1(\boldsymbol{x})\big) \\
	&= R_{\overline{\Delta}}\big(\overline{\Sigma}_2^{\std}(R_{\overline{\Delta}}\phi\pi_2(\boldsymbol{x}))\big) \cap R_{\overline{\Delta}}\big(\overline{\Sigma}_1^{\std}(R_{\overline{\Delta}}\phi\pi_1(\boldsymbol{x}))\big) \\
	&= R_{\overline{\Delta}}\big(\overline{\Sigma}_2^{\std}(\phi\pi_2(\boldsymbol{x}))\big) \cap R_{\overline{\Delta}}\big(\overline{\Sigma}_1^{\std}(\phi\pi_1(\boldsymbol{x}))\big) \\
	&= R_{\overline{\Delta}}\Big[\overline{\Sigma}_2^{\std}(\phi\pi_2(\boldsymbol{x}))\cap\overline{\Sigma}_1^{\std}(\phi\pi_1(\boldsymbol{x}))\Big] \\
	&= R_{\overline{\Delta}}(\Phi(\boldsymbol{x})).
\end{align*}
This completes the lemma.
\end{proof}

\subsection{Equivariant Moser trick}
The diffeomorphism $\Phi$ does not necessarily preserve the symplectic form $\ow$, so we employ the Moser's trick to adjust the pull-back $(\Phi^{-1})^*\ow$ while keeping its compatibility with the involutions.
\begin{lemma}\label{lem: Moser}
	There exists a diffeomorphism $\Psi\colon Q\to Q$ such that
\begin{itemize}
	\item[(i)] $\Psi^*\ow=(\Phi^{-1})^*\ow$.
	\item[(ii)] $\Psi\circ R_{\overline{\Delta}}=R_{\overline{\Delta}}\circ \Psi$.
	\item[(iii)] $\Psi$ acts trivially on homology.\end{itemize}
\end{lemma}
Consider a smooth family of closed 2-forms on $Q$ given by
$$
\ow_t:=(1-t)(\Phi^{-1})^*\ow+t\ow,\qquad t\in [0,1].
$$
\begin{lemma}
For each $t\in [0,1]$ we have
\begin{enumerate}
	\item[(i)]\label{eq: (i)} $\ow_t$ is a symplectic form.
	\item[(ii)] $[\ow_t]=[\ow]$ in $H_{\dR}^2(Q)$.
	\item[(iii)] $R_{\overline{\Delta}}^*\ow_t=-\ow_t$.
\end{enumerate}	
\end{lemma}
\begin{proof}
To show (i), we claim that $\ow_t$ is compatible with $(\Phi^{-1})^*J$, that is, $\ow_t(X,(\Phi^{-1})^*JX)>0$ for all non-zero $X\in TQ$.
By naturality, we have $(\Phi^{-1})^*\ow(X,(\Phi^{-1})^*JX)>0$.
Since the Gromov's foliations associated to $(\Phi^{-1})^*J$ coincide with the standard foliations $\mathcal{F}_i^{\std}$ (in particular, the tangent bundle $TQ$ admits a decomposition into two transversal $(\Phi^{-1})^*J$-invariant subbundles which are $\ow$-orthogonal), we deduce $\ow(X,(\Phi^{-1})^*JX)>0$.
Moreover, $(\Phi^{-1})^*\ow$ evaluates to 1 over a leaf of $\mathcal{F}_i^{\std}$ for $i=1,2$, so $(\Phi^{-1})^*[\ow]=[\ow]\in H^2_{\dR}(Q)$.
This implies (ii).
By Lemma~\ref{lem: Phi_equiv}, (iii) follows.
\end{proof}
We are in position to apply the Moser trick, see \cite[Theorem~3.2.4]{MS}.
Since $[\ow_1]=[\ow_0]$, there exists a 1-form $\beta$ on $Q$ such that $\ow_1-\ow_0=d\beta$.
Then $\ow_t=(\Phi^{-1})^*\ow + td\beta$.
Since $R_{\overline{\Delta}}^*d\beta=-d\beta$, the anti-averaged 1-form $\tilde{\beta}:=\frac{1}{2}(\beta-R_{\overline{\Delta}}^*\beta)$ satisfies $\ow_t=(\Phi^{-1})^*\ow+td\tilde{\beta}$ and $R_{\overline{\Delta}}^*\tilde{\beta}=-\tilde{\beta}$.
Let $X_t$ be the Moser vector field associated to $\tilde{\beta}$, that is, the unique solution of the equation $\ow_t(X_t,\cdot)+\tilde{\beta}=0$.
The flow $\Psi_t$ of $X_t$ satisfies $\Psi_t^*\ow_t=\ow_0$ and $\Psi_t\circ R_{\overline{\Delta}}=R_{\overline{\Delta}}\circ \Psi_t$ for all $t\in [0,1]$.
Noting that $R_{\overline{\Delta}}^*X_t=X_t$, the equation
$$
\frac{d}{dt}(R_{\overline{\Delta}}\circ \Psi_t\circ R_{\overline{\Delta}})=dR_{\overline{\Delta}}\Big(\frac{d}{dt}(\Psi_t\circ R_{\overline{\Delta}})\Big)=dR_{\overline{\Delta}}\big(X_t(\Psi_t\circ R_{\overline{\Delta}})\big)=X_t(R_{\overline{\Delta}}\circ \Psi_t\circ R_{\overline{\Delta}}),
$$
and the uniqueness of the solution show that $\Psi_t$ and $R_{\overline{\Delta}}$ commute.
The time-1 flow $\Psi:=\Psi_1$ satisfies the properties (i)--(iii).
The condition (iii) follows since $\Psi$ is isotopic to the identity.
This completes the proof of Lemma~\ref{lem: Moser}.
\begin{proof}[Proof of the Main Theorem.]
Note that $\Gamma=\Psi\circ \Phi\in \Symp(Q,\ow)$ acts trivially on homology and that $\Gamma\circ R=R_{\overline{\Delta}}\circ \Gamma$.
By Gromov's theorem, $\Gamma$ is symplectically isotopic to the identity, and hence $\Gamma\in \Ham(Q,\ow)$ as $H^1_{\dR}(Q)=0$.
This shows that $R\in \A(Q,\ow)$ is Hamiltonian isotopic to $R_{\overline{\Delta}}$ as sought.
\end{proof}

\subsection*{Acknowledgement}
The authors cordially thank Viatcheslav Kharlamov and Vsevolod Shevchishin for kind explanation of their work.
JM specially thanks her advisor Suyoung Choi for continued support and encouragement.
JK is supported by a KIAS Individual Grant SP068001 via the Center for Mathematical Challenges at Korea Institute for Advanced Study and by the POSCO Science Fellowship of POSCO TJ Park Foundation.
JM is supported by NRF-2019R1A2C2010989 and Samsung Science and Technology Foundation under Project Number SSTF-BA1901-01.

\bibliographystyle{abbrv}
\bibliography{mybibfile}
\noindent
\footnotesize{Korea Institute for Advanced Study (KIAS), 85 Hoegiro, Dongdaemun-gu, Seoul 02455, Republic of Korea\vspace{0.1cm}\\
\emph{E-mail address: }\texttt{joontae@kias.re.kr}\vspace{0.3cm}\\
Department of Mathematics, Ajou University, 206 Worldcup-ro, Suwon 16499, Republic of Korea\vspace{0.1cm}\\
\emph{E-mail address: }\texttt{j9746@ajou.ac.kr}}

\end{document}